\documentclass[a4paper, 11pt]{article}
\usepackage{amsmath}
\usepackage{amssymb}
\usepackage{amsthm}
\usepackage{amsfonts}
\usepackage{latexsym}
\usepackage{mathtools}
\usepackage{empheq}
\usepackage[dvipdfmx]{graphicx}
\usepackage{comment}
\usepackage{cases}
\usepackage{float}
\usepackage{color}
\usepackage[top=26truemm, bottom=25truemm, left=25.5truemm, right=25.5truemm]{geometry}
\usepackage{titlefoot}
\usepackage{subcaption}

\allowdisplaybreaks[1]

\newcommand{\capR}{\mathbb{R}}

\newcommand{\footremember}[2]{%
	\footnote{#2}
	\newcounter{#1}
	\setcounter{#1}{\value{footnote}}%
}
\newcommand{\footrecall}[1]{%
	\footnotemark[\value{#1}]%
} 

\let\OLDthebibliography\thebibliography
\renewcommand\thebibliography[1]{
	\OLDthebibliography{#1}
	\setlength{\parskip}{1.5pt}
	\setlength{\itemsep}{1pt plus 0.3ex}
}

\makeatletter

\@addtoreset{equation}{section}
\makeatother

\theoremstyle{mystyle}
\newtheorem{theorem}{Theorem}[section]
\newtheorem*{theorem*}{Theorem}

\newtheorem{proposition}[theorem]{Proposition}

\theoremstyle{definition}

\theoremstyle{remark}

\begin{document}
\title{(Dis)connectedness of nonlocal minimal surfaces in a cylinder and a stickiness property}

\author{%
	Serena Dipierro\footremember{alley}{Department of Mathematics and Statistics, University of Western Australia, 35 Stirling Hwy, Crawley WA 6009, Australia. E-mail: {\tt serena.dipierro@uwa.edu.au}}%
	\and Fumihiko Onoue\footremember{trailer}{Scuola Normale Superiore, Piazza dei Cavalieri 7, Pisa 56126, Italy. E-mail: {\tt fumihiko.onoue@sns.it}}%
	\and Enrico Valdinoci\footrecall{alley} \footnote{E-mail: {\tt enrico.valdinoci@uwa.edu.au}}%
}

\date{\today}

\maketitle

\begin{abstract}
We consider nonlocal minimal surfaces in a cylinder with prescribed
datum given by the complement of a slab. We show that when
the width of the slab is large the minimizers are disconnected and
when the width of the slab is small the minimizers are connected.
This feature is in agreement with the classical case of the minimal surfaces.

Nevertheless, we show that when
the width of the slab is large the minimizers are not flat discs,
as it happens in the classical setting, and, in particular,
in dimension~$2$ we provide a quantitative bound on
the stickiness property exhibited by the minimizers.

Moreover, differently from the classical case,
we show that when the width of the slab is small then the minimizers
completely adhere to the side of the cylinder, thus providing a further
example of stickiness phenomenon.
\end{abstract}

\section{Introduction}

Nonlocal minimal surfaces were introduced in~\cite{CRS}
and constitute one of the most fascinating, and challenging,
research topics in the realm of fractional equations. Roughly speaking,
the problem is that of minimizing an energy functional
built by the pointwise interaction of a set versus its complement
(this energy functional can also be conveniently ``localized''
in a given domain by taking into account the interactions
in which at least one point lies in the domain).
The prototype interaction taken into account is scaling and translation invariant and
with polynomial decay
(but we mention that other versions of the problem considered also interactions
via integrable kernels, see~\cite{MAZ1, MAZ2, CINSE}).

The nonlocal minimal surfaces constructed by this minimization procedure
have relevant features in terms of differential geometry
and geometric measure theory, since their energy functional
can be considered as a nonlocal approximation of the classical perimeter
functional and the nonlocal minimal surfaces as a fractional variant of
the classical minimal surfaces, see
\cite{BOU, DAVI, PON, CA1, AMB, CA2}.
Critical points of the nonlocal perimeter energy functional satisfy
an integral relation that can be seen as a vanishing nonlocal mean curvature prescription
(see~\cite{CRS, NOTI, LAW, CINDEL}) and accordingly the study
of volume prescribed minimizers leads to the analysis of surfaces with constant
nonlocal mean curvature (see~\cite{DAVN, CAB1, CAB2, CIRA}).
Moreover,
nonlocal minimal surfaces arise as the large-scale limit
of long-range phase coexistence models (see \cite{SAVA}),
as discrete iterations of fractional heat equations (see~\cite{SOU})
and as continuous approximations
of interfaces of long-range Ising models (see \cite{COZVA}).
\medskip

Given the importance of nonlocal minimal surfaces
from all these perspectives, it is desirable to develop some intuition
about their basic geometric features. For this,
since it is very rare to have explicit solutions and precise formulas
which entirely describe nonlocal minimal surfaces, it is often convenient
to focus on some simplified cases in which the reference domain and the external
datum possess some special characteristics which lead to a deep understanding
of at least some cardinal aspects of the object under investigation.\medskip

This note follows precisely in this line of research, namely we will consider
a very simple domain, that is a vertical cylinder in~$\capR^n$, and a very special
external datum, that is the complement of
a horizontal slab, and detect how the minimizers of the nonlocal perimeter functional
change when the width of the slab varies.\medskip

On the one hand, when the width of the slab is large, we will show that
these minimizers are disconnected, and this is somehow the nonlocal counterpart
of the fact that the classical perimeter gets minimized by far-away parallel and co-axial discs.

On the other hand,
when the width of the slab is small, we will show that
these minimizers become connected. This change of topology is in agreement with the classical case,
since perimeter minimizers constrained to two nearby parallel and co-axial circumferences
are connected necks of catenoids. Nonetheless, the specific
geometry exhibited in this case by nonlocal minimal surfaces is rather different
from
that of catenoids, since we will additionally show that when the width of the slab is small
the nonlocal minimal surface obtained with this procedure actually coincides inside the cylinder with the
cylinder itself.\medskip

More precisely, and in further detail,
the mathematical framework that we use in this paper is the
one
introduced in \cite{CRS} and can be summarized as follows.
Let $s\in(0,\,1)$ and $\Omega \subset \capR^n$ be an open subset
with Lipschitz boundary. Then we define the nonlocal perimeter or $s$-perimeter $P_s(E\,;\,\Omega)$ for a measurable set $E\subset \capR^n$ by
\begin{equation}\label{defiNonlocalPeri}
P_s(E\,;\,\Omega) := \int_{E\cap \Omega}\int_{E^c}\frac{dx\,dy}{|x-y|^{n+s}}+
\int_{E\cap \Omega^c}\int_{\Omega\cap E^c}\frac{dx\,dy}{|x-y|^{n+s}},
\end{equation}
where we denote by $E^c$ the complement of a set $E$.
We say that a set $E \subset \capR^n$ is a $s$-minimizer or $s$-minimal set in $\Omega$ if it holds that $P_s(E\,;\Omega')
\leq P_s(F\,;\Omega')$ for any open, bounded, and Lipschitz set~$\Omega'$ contained in~$\Omega$
and any~$F \subset \capR^n$ with $F \setminus \Omega' = E \setminus \Omega'$. 
See also~\cite{LOM} for additional details regarding the minimization procedure
in bounded or unbounded domains.\medskip

For our purposes,
we will often
denote coordinates in~$\capR^n$ by~$x=(x',x_n)\in\capR^{n-1}\times\capR$
and we will focus here on the case of ``cylindrical'' domains of the form
\begin{equation}\label{omega}
\Omega := \{x=(x',x_n)\in\capR^{n-1}\times\capR {\mbox{ s.t. }} |x'|<1\}.
\end{equation}
We are interested in sets~$E$ whose exterior prescription outside~$\Omega$
is the complement of a strip. Namely, given~$M>0$, we define
\begin{equation}\label{E0}
E_0:= \{x=(x',x_n)\in\capR^{n-1}\times\capR {\mbox{ s.t. }} |x_n|>M\}
\end{equation}
and we consider $s$-minimal sets in~$\Omega$ such that~$E\setminus \Omega=E_0\setminus\Omega$. See e.g.~\cite[Theorem 0.2.5]{LOM}
for existence results for this type of $s$-minimal sets.

Our main concern in this note is how the variation of the parameter~$M$
affects the topological property of the $s$-minimizer and we will show that
{\em for small values of~$M$ the $s$-minimizer is connected} while {\em for
large values it is disconnected}. 

Furthermore, we will show that {\em for small values of~$M$
the $s$-minimizer in~$\Omega$ coincides with~$\Omega$ itself},
and this is an interesting difference with respect to the classical case of minimal surfaces.
Indeed, when~$n\ge3$ minimal surfaces in a cylinder do not coincide with the cylinder itself
and, when connected, they develop a ``neck'' inside the cylinder, as exhibited by the classical example of the
catenoid (as a matter of fact, when~$n\ge3$
the cylinder does not have vanishing mean curvature, hence
it cannot be a minimizer for the classical perimeter functional).

Therefore, our construction of nonlocal minimal surfaces
that coincide with the cylinder in their free domain
heavily relies on the nonlocal character of the problem taken into consideration
and can be seen as a new example of the {\em stickiness theory for nonlocal minimal surfaces}
which was introduced in~\cite{DSVboundary} and developed in~\cite{DSVGraphProp, BLV, DSVminiGraph, DSV3D}.
See also~\cite{DSVDGB, 234DZ} for surveys on nonlocal minimal surfaces
discussing, among other topics, the stickiness phenomenon (and, for instance~\cite{HaSi} to appreciate the structural
differences with respect to the classical case).\medskip

In further detail, the precise result that we have concerning the connectedness
of the $s$-minimizer and its stickiness properties for small values of~$M$
goes as follows:

\begin{figure}
	\begin{subfigure}{.5\textwidth}
		\centering
		\includegraphics[width=.8\linewidth]{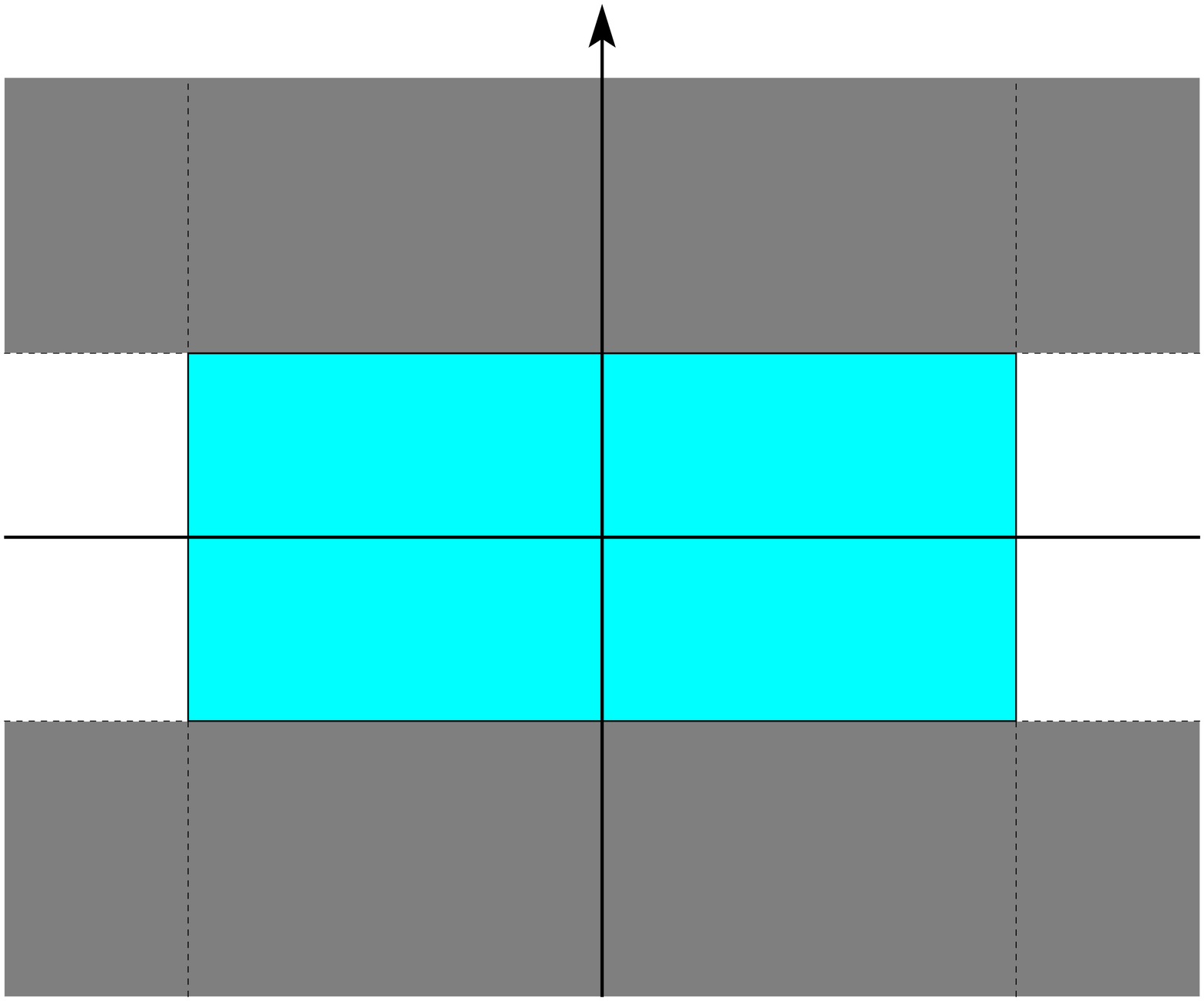}
		\caption{}
		\label{figure-0-MINI1}
	\end{subfigure}%
	\begin{subfigure}{.5\textwidth}
		\centering
		\includegraphics[width=.8\linewidth]{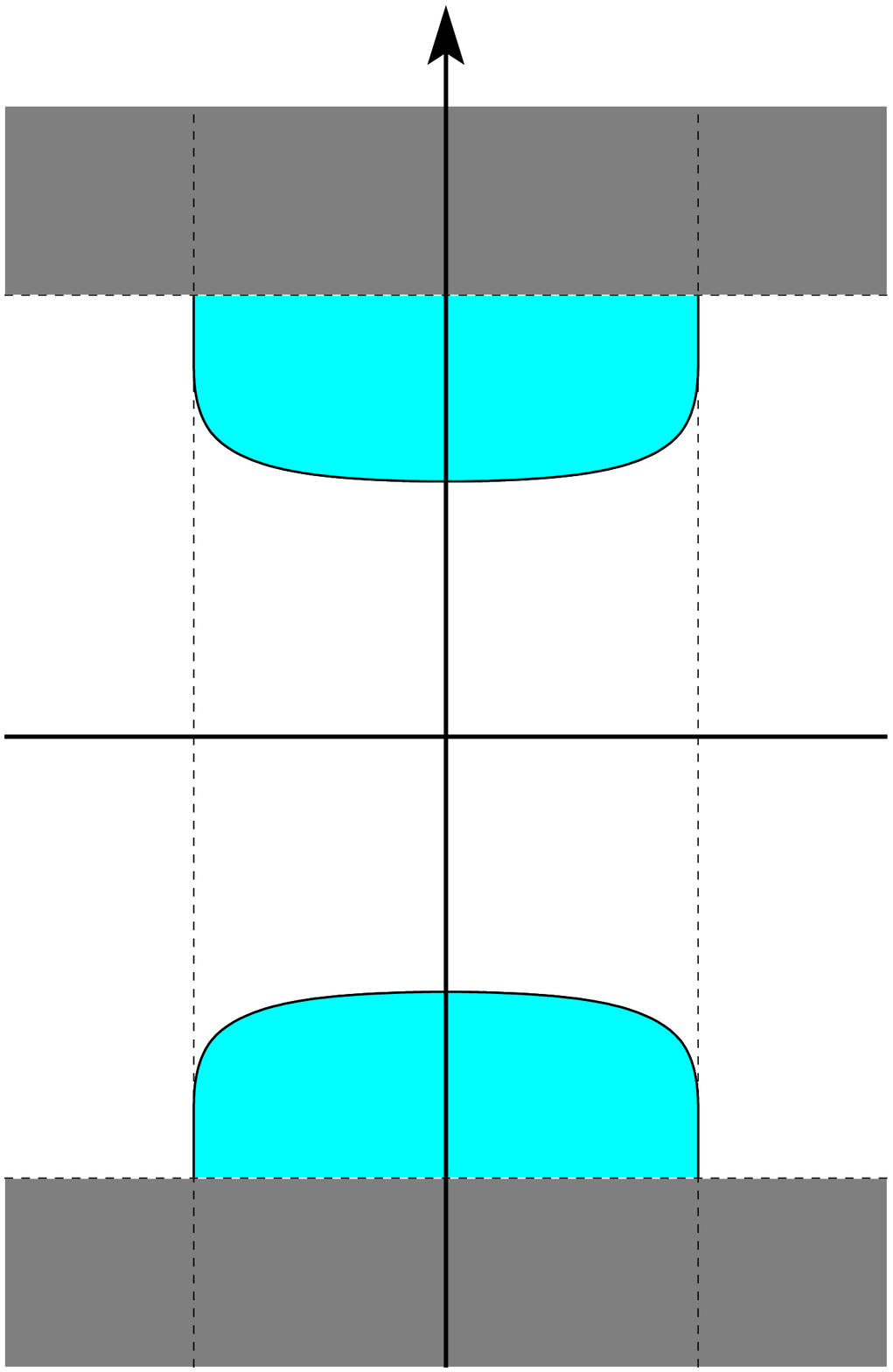}
		\caption{}
		\label{figure-0-MINI2}
	\end{subfigure}
	\caption{The minimizers in Theorem~\ref{mainTheorem} (left) and Theorem~\ref{mainTheorem02} (right).}
\end{figure}

\begin{theorem}\label{mainTheorem}
	Let $\Omega$ be as in~\eqref{omega}
	and let $E_0$ be
	defined by \eqref{E0}. Then, there exists~$M_0\in(0,1)$, depending
	only on~$n$ and~$s$, such that,
	for any $M\in(0,\,M_0)$, the minimizer~$E_M$ in~$\Omega$ of~$P_s$
	coincides with~$\Omega$. In particular, $E_M$ is connected.
\end{theorem}

The minimizer described in
Theorem~\ref{mainTheorem}
is depicted in Figure~\ref{figure-0-MINI1}.
As a counterpart
of Theorem~\ref{mainTheorem}, the disconnectedness result for large values of $M$
is the following:

\begin{theorem}\label{mainTheorem02}
	Let $\Omega$ be as in~\eqref{omega}
	and let $E_0$ be
	defined by \eqref{E0}.
Then, there exists~$M_0>1$, depending only on~$n$ and~$s$, such that,
for any~$M > M_0$, the minimizer~$E_M$ in~$\Omega$ of~$P_s$
is disconnected.
\end{theorem}

To favor the intuition,
a sketch on how we believe the minimizer
in Theorem~\ref{mainTheorem02}
looks like is given in Figure~\ref{figure-0-MINI2}.

Interestingly, the situation described in
Theorem~\ref{mainTheorem02} is
similar, but structurally different from the one exhibited by
classical minimal surfaces. Indeed,
the analogy with the classical case is given by
the disconnectedness of the minimizers. The difference
in the pattern is that classical minimal surfaces
in the framework of Theorem~\ref{mainTheorem02}
are just flat disc, and this is not the
case for their corresponding nonlocal counterpart
(as we will make
precise in Proposition~\ref{KSPN-t3igkrjjghNNSND}).

The forthcoming Sections~\ref{S:002} and~\ref{S:003}
contain the proofs of
Theorems~\ref{mainTheorem} and~\ref{mainTheorem02} respectively.
In Section~\ref{HSI9urj-90369iyktpogmn}
we will present further similarities and differences
with respect to the classical case in the framework
of large~$M$ given by
Theorem~\ref{mainTheorem02}.

\section{Proof of Theorem~\ref{mainTheorem}}\label{S:002}

Let~$E_M$ be the minimizer
selected in Theorem~\ref{mainTheorem}, see Figure~\ref{figure01}
(at this stage of the proof, we do not really know how this minimizer
looks like, so the one depicted in Figure~\ref{figure01}
will not be the ``real'' minimizer after all).

\begin{figure}
	\begin{center}
		\includegraphics[keepaspectratio,scale=0.55,angle=0]{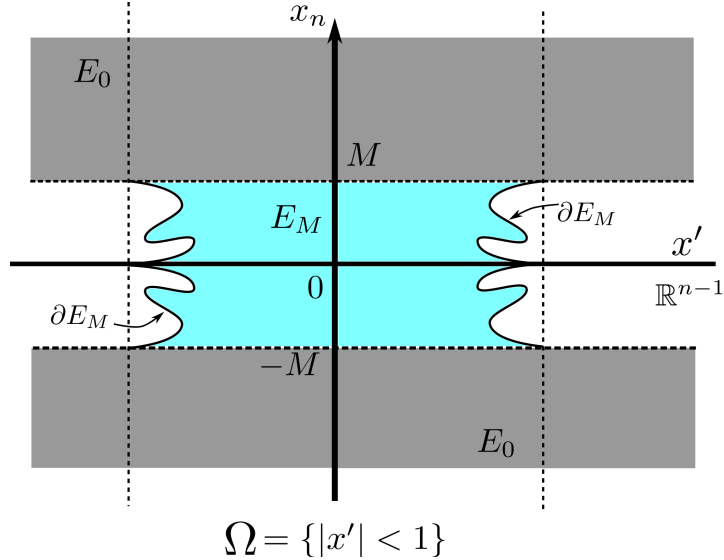}
		\caption{The situation in the proof of
			Theorem~\ref{mainTheorem}.}
		\label{figure01}
	\end{center}
\end{figure}

By~\cite[Corollary 5.3]{CRS}, we know that
\begin{equation}\label{f1}
\{ x_n>M\}\cup\{x_n<-M\}\subset E_M.
\end{equation}
Given~$t\in\capR$ and~$r\in(0,1)$,
we consider the ball of radius $r$ with center $te_n$,
where~$e_n=(0,\dots,0,1)$.
By~\eqref{f1}, we have that~$B_r(te_n)\subset E_M$ for every~$t>M+1$.
Hence, we can slide such a ball downwards till it touches~$\partial E_M$
inside~$\Omega$. The content of Theorem~\ref{mainTheorem}
is precisely that this touching does not occur, hence, by
contradiction, we suppose instead that there exist~$t_0\in\capR$
and~$r_0\in(0,1)$
such that
\begin{equation}\label{F44}
B_{r_0}(te_n)\subset E_M \quad {\mbox{for all }}t>t_0\end{equation}
with
\begin{equation}
		\partial B_{r_0}(t_0e_n) \cap \partial E_{M}  \neq \varnothing.
	\end{equation}
	Then, setting $z:=t_0 e_n$, we can choose a point $q=(q',q_n) \in \partial B_{r_0}(z) \cap \partial E_{M}$.
			
	Since $E_M$ is a local minimizer of $P_s$ in $\Omega$,
	we obtain, by using the Euler-Lagrange equation in the viscosity sense shown in \cite[Theorem 5.1]{CRS} (see also~\cite[Theorem B.9]{BLV}), that
	\begin{equation}\label{euler_lagrange_fractional}
		\int_{\capR^n}\frac{\chi_{E_M^c}(y)-\chi_{E_M}(y)}{|y-q|^{n+s}}\,dy \geq 0.
	\end{equation}
Our goal is now to produce a contradiction with~\eqref{euler_lagrange_fractional}
by showing that the left hand side is strictly negative.
To this end,
we let
$$ S_M:=\capR^{n-1}\times[q_n-2M,q_n+2M].$$
We remark that
\begin{equation}\label{F3}
E^c_M\subset S_M\setminus B_{r_0}(z).
\end{equation}
Indeed, by~\eqref{f1} we know that~$q_n\in[-M,M]$
and~$E_M^c\subset\{x_n\in[-M,M]\}$, whence~$E_M^c\subset S_M$.
This and~\eqref{F44} give~\eqref{F3}.

We also observe that~$S_M\supset \{|x_n|\le M\}$, and therefore,
in light of~\eqref{f1},
\begin{equation}\label{F66}
S_M^c\subset E_M.
\end{equation}
Moreover, using the change of variable~$y\mapsto y+q$,
\begin{equation}\label{KJMSi0ujf9htgf}
\int_{S^c_M}\frac{dy}{|y-q|^{n+s}}=
\int_{\capR^{n-1}\times((-\infty,-2M)\cup(2M,+\infty))}\frac{dy}{|y|^{n+s}}\ge
\int_{B_M(3Me_n)}\frac{dy}{|y|^{n+s}}\ge cM^{-s},
\end{equation}
for a constant~$c>0$ depending only on~$n$.

\begin{figure}
	\begin{center}
		\includegraphics[keepaspectratio,scale=0.55,angle=0]{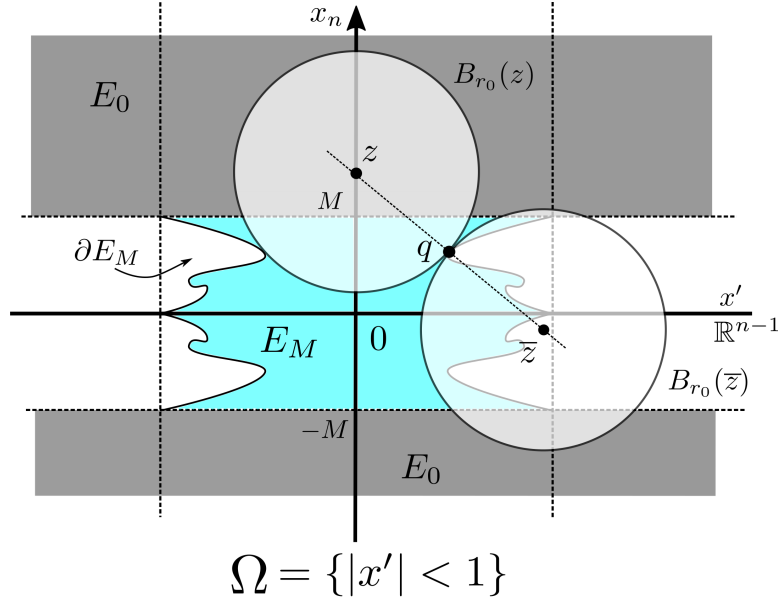}
		\caption{The touching between the ball~$B_{r_0}(z)$ and
		the symmetric ball~$B_{r_0}( \overline{z})$ at the point~$q$.}
		\label{figure02}
	\end{center}
\end{figure}

Now we set~$\overline{z}:=z+2(q-z)$ and we
consider the symmetric ball~$B_{r_0}(\overline{z})$ with respect to~$q$,
see Figure~\ref{figure02}.
Moreover, we take a free parameter~$\Lambda\ge4$, to be chosen conveniently large in what follows
and we observe that, by symmetry,
$$ \int_{S_M\cap B_{\Lambda M}(q)\cap B_{r_0}(z)}\frac{dy}{|y-q|^{n+s}}=
\int_{S_M\cap B_{\Lambda M}(q)\cap B_{r_0}(\overline z)}\frac{dy}{|y-q|^{n+s}}.$$
Also, by~\eqref{F3},
$$ \int_{S_M\cap B_{\Lambda M}(q)\cap B_{r_0}(z)}\frac{\chi_{E_M^c}(y)-\chi_{E_M}(y)}{|y-q|^{n+s}}\,dy
=-\int_{S_M\cap B_{\Lambda M}(q)\cap B_{r_0}(z)}\frac{dy}{|y-q|^{n+s}},$$
and consequently
\begin{eqnarray*}
&& \int_{S_M\cap B_{\Lambda M}(q)\cap B_{r_0}(z)}\frac{\chi_{E_M^c}(y)-\chi_{E_M}(y)}{|y-q|^{n+s}}\,dy+
\int_{S_M\cap B_{\Lambda M}(q)\cap B_{r_0}(\overline z)}\frac{\chi_{E_M^c}(y)-\chi_{E_M}(y)}{|y-q|^{n+s}}\,dy
\\ &&\qquad\le
-\int_{S_M\cap B_{\Lambda M}(q)\cap B_{r_0}(z)}\frac{dy}{|y-q|^{n+s}}+
\int_{S_M\cap B_{\Lambda M}(q)\cap B_{r_0}(\overline z)}\frac{dy}{|y-q|^{n+s}}=0.
\end{eqnarray*}
Therefore,
\begin{equation}\label{f098}
\begin{split}&
\int_{S_M\cap B_{\Lambda M}(q)}\frac{\chi_{E_M^c}(y)-\chi_{E_M}(y)}{|y-q|^{n+s}}\,dy\\
=\;&
\int_{S_M\cap B_{\Lambda M}(q)\cap B_{r_0}(z)}\frac{\chi_{E_M^c}(y)-\chi_{E_M}(y)}{|y-q|^{n+s}}\,dy+
\int_{S_M\cap B_{\Lambda M}(q)\cap B_{r_0}(\overline z)}\frac{\chi_{E_M^c}(y)-\chi_{E_M}(y)}{|y-q|^{n+s}}\,dy
\\&\qquad+\int_{S_M\cap \big(B_{\Lambda M}(q)\setminus
\big( B_{r_0}(z)\cup B_{r_0}(\overline z)\big)\big)
}\frac{\chi_{E_M^c}(y)-\chi_{E_M}(y)}{|y-q|^{n+s}}\,dy\\ \le\;&
\int_{S_M\cap \big(B_{\Lambda M}(q)\setminus
\big( B_{r_0}(z)\cup B_{r_0}(\overline z)\big)\big)
}\frac{\chi_{E_M^c}(y)-\chi_{E_M}(y)}{|y-q|^{n+s}}\,dy\\ \le\;& \int_{ B_{\Lambda M}(q)\setminus
\big( B_{r_0}(z)\cup B_{r_0}(\overline z)\big)
}\frac{dy}{|y-q|^{n+s}}\\ \le\;&C\Lambda^{1-s} M^{1-s},
\end{split}
\end{equation}
for some~$C>0$ depending only on~$n$ and~$s$,
where~\cite[Lemma 3.1]{DSVGraphProp}
has been used in the last inequality (here with~$R:=1$ and~$\lambda:=\Lambda M$).

Furthermore,
\begin{eqnarray*}&&
\int_{S_M\setminus B_{\Lambda M}(q)}\frac{\chi_{E_M^c}(y)-\chi_{E_M}(y)}{|y-q|^{n+s}}\,dy\le
\int_{S_M\setminus B_{\Lambda M}(q)}\frac{dy}{|y-q|^{n+s}}
\\&&\qquad=\int_{(\capR^{n-1}\times[-2M,2M])\setminus B_{\Lambda M}}\frac{dy}{|y|^{n+s}}\le
\int_{(\capR^{n-1}\times[-2M,2M])\setminus B_{\Lambda M}}\frac{dy}{|y'|^{n+s}}\\&&\qquad\le
\int_{\{|y'|\ge\Lambda M/2,\;|y_n|\le 2M\}}\frac{dy}{|y'|^{n+s}}=\frac{C_0}{\Lambda^{1+s}M^s},
\end{eqnarray*}
for some~$C_0>0$ depending only on~$n$ and~$s$.

Hence, combining this information with~\eqref{f098},
$$ \int_{S_M}\frac{\chi_{E_M^c}(y)-\chi_{E_M}(y)}{|y-q|^{n+s}}\,dy\le
C\Lambda^{1-s} M^{1-s}+\frac{C_0}{\Lambda^{1+s}M^s}.
$$
This, \eqref{F66} and~\eqref{KJMSi0ujf9htgf} lead to
\begin{eqnarray*}&&
\int_{\capR^n}\frac{\chi_{E_M^c}(y)-\chi_{E_M}(y)}{|y-q|^{n+s}}\,dy\\
&=&
-\int_{S_M^c}\frac{dy}{|y-q|^{n+s}}
+\int_{S_M}\frac{\chi_{E_M^c}(y)-\chi_{E_M}(y)}{|y-q|^{n+s}}\,dy\\&\le&
-cM^{-s}
+C\Lambda^{1-s} M^{1-s}+\frac{C_0}{\Lambda^{1+s}M^s}\\&=&
-cM^{-s}\left(1-\frac{C\Lambda^{1-s} M }{c}-\frac{C_0}{c\Lambda^{1+s}}
\right).
\end{eqnarray*}
Now we choose~$\Lambda:=\max\left\{4,
\left(\frac{2C_0}{c}\right)^{\frac1{1+s}}\right\}$ and we thus obtain that
$$ \int_{\capR^n}\frac{\chi_{E_M^c}(y)-\chi_{E_M}(y)}{|y-q|^{n+s}}\,dy\le
-cM^{-s}\left(\frac12-\frac{C\Lambda^{1-s} M }{c}
\right).$$
Taking now~$M$ conveniently small, we conclude that
$$ \int_{\capR^n}\frac{\chi_{E_M^c}(y)-\chi_{E_M}(y)}{|y-q|^{n+s}}\,dy\le
-\frac{cM^{-s}}4<0,$$
which produces the desired contradiction with~\eqref{euler_lagrange_fractional}.

\section{Proof of Theorem~\ref{mainTheorem02}}\label{S:003}

We let~$M>1$ to be chosen conveniently large.
Given~$t\in\capR$, we consider the ball~$B_{\sqrt{M}}( t e_1)$, where~$e_1=(1,0,\dots,0)$,
and we slide it from left to right till it touches~$\partial E_M$.
Notice indeed that~$B_{\sqrt{M}}(te_1)\subset E_0^c$ when~$t<-\sqrt{M}$
and, to prove Theorem~\ref{mainTheorem02}, we suppose by contradiction that there exists~$t_0\in\capR$ such
that~$B_{\sqrt{M}}(te_1)\subset E_M^c$ for all~$t<t_0$
with~$\partial B_{\sqrt{M}}(t_0e_1)\cap\partial E_M\ne\varnothing$.

We set~$z:=t_0e_1$ and we pick a point~$q=(q',q_n)\in\partial B_{\sqrt{M}}(z)\cap\partial E_M$.
By the Euler-Lagrange equation in the viscosity sense (see \cite[Theorem 5.1]{CRS} and \cite[Theorem B.9]{BLV}), we know that
\begin{equation}\label{ihakjb-iknnsbab66}
\int_{\capR^n}\frac{\chi_{E_M^c}(y)-\chi_{E_M}(y)}{|y-q|^{n+s}}\,dy\le0.
\end{equation}
We consider the symmetric ball with respect to~$q$, by defining~$\overline z:=z+2(q-z)$
and taking into account the ball~$B_{\sqrt{M}}(\overline{z})$,
see Figure~\ref{figure03}.

\begin{figure}
	\begin{center}
		\includegraphics[keepaspectratio,scale=0.60,angle=0]{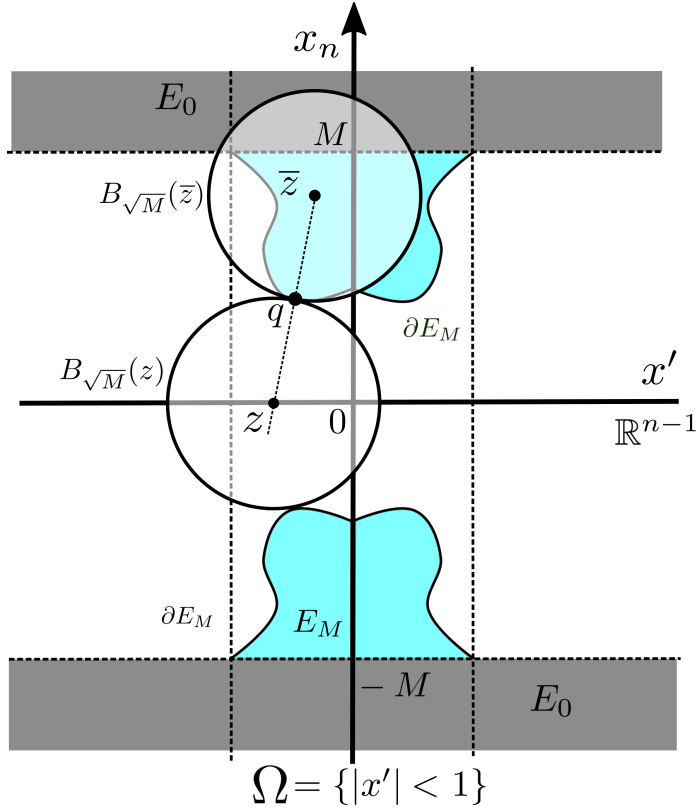}
		\caption{The touching between the ball~$B_{\sqrt{M}}(z)$ and
			the symmetric ball~$B_{\sqrt{M}}( \overline{z})$
			at the point~$q$.}
		\label{figure03}
	\end{center}
\end{figure}

We define
$$ S:=\big\{ x=(x',x_n)\in\capR^{n-1}\times\capR
{\mbox{ s.t. }}|x'-q'|\le 3 
\big\}.$$
By symmetry,
$$ \int_{S\cap B_{\sqrt{M}}(z)}\frac{dy}{|y-q|^{n+s}}=
\int_{S\cap B_{\sqrt{M}}(\overline z)}\frac{dy}{|y-q|^{n+s}}$$
and therefore
\begin{equation}\label{04msaolnc}
\begin{split}&
\int_{S}\frac{\chi_{E_M^c}(y)-\chi_{E_M}(y)}{|y-q|^{n+s}}\,dy\\
=\;&
\int_{S\cap B_{\sqrt{M}}(z)}\frac{\chi_{E_M^c}(y)-\chi_{E_M}(y)}{|y-q|^{n+s}}\,dy
+
\int_{S\cap B_{\sqrt{M}}(\overline z)}\frac{\chi_{E_M^c}(y)-\chi_{E_M}(y)}{|y-q|^{n+s}}\,dy
\\&\qquad+
\int_{S\setminus\big( B_{\sqrt{M}}(z)
\cup B_{\sqrt{M}}(\overline z)\big)}\frac{\chi_{E_M^c}(y)-\chi_{E_M}(y)}{|y-q|^{n+s}}\,dy\\
\ge\;&
\int_{S\cap B_{\sqrt{M}}(z)}\frac{dy}{|y-q|^{n+s}}-
\int_{S\cap B_{\sqrt{M}}(\overline z)}\frac{dy}{|y-q|^{n+s}}
\\&\qquad+
\int_{S\setminus\big( B_{\sqrt{M}}(z)
\cup B_{\sqrt{M}}(\overline z)\big)}\frac{\chi_{E_M^c}(y)-\chi_{E_M}(y)}{|y-q|^{n+s}}\,dy\\
\ge\;&-
\int_{S\setminus\big( B_{\sqrt{M}}(z)
\cup B_{\sqrt{M}}(\overline z)\big)}\frac{dy}{|y-q|^{n+s}}.
\end{split}
\end{equation}
Now, in view of~\cite[Lemma 3.1]{DSVGraphProp}, used here with~$R:=\sqrt{M}$
and~$\lambda:=1/\sqrt[4]{M}$, we know that
\begin{equation*}
\int_{B_{\sqrt[4]{M}}(q)\setminus\big( B_{\sqrt{M}}(z)
\cup B_{\sqrt{M}}(\overline z)\big)}\frac{dy}{|y-q|^{n+s}}\le C M^{-\frac{1+s}{4}},
\end{equation*}
for some~$C>0$ depending only on~$n$ and~$s$.
As a result,
\begin{eqnarray*}
&&\int_{S\setminus\big( B_{\sqrt{M}}(z)
\cup B_{\sqrt{M}}(\overline z)\big)}\frac{dy}{|y-q|^{n+s}}
\le
\int_{B_{\sqrt[4]{M}}(q)\setminus\big( B_{\sqrt{M}}(z)
\cup B_{\sqrt{M}}(\overline z)\big)}\frac{dy}{|y-q|^{n+s}}
+
\int_{S\setminus B_{\sqrt[4]{M}}(q)}\frac{dy}{|y-q|^{n+s}}\\
&&\qquad\le C M^{-\frac{1+s}{4}} +\int_{\capR^n\setminus B_{\sqrt[4]{M}}(q)}\frac{dy}{|y-q|^{n+s}}
= C M^{-\frac{1+s}{4}}+C_1 M^{-\frac{s}{4}}\le C_2 M^{-\frac{s}{4}} ,
\end{eqnarray*}
for some~$C_1>0$ depending only on~$n$ and~$s$, with~$C_2:=C+C_1$.

This and~\eqref{04msaolnc} lead to
\begin{eqnarray*}&&
\int_{\capR^n}\frac{\chi_{E_M^c}(y)-\chi_{E_M}(y)}{|y-q|^{n+s}}\,dy\\&=&
\int_{S}\frac{\chi_{E_M^c}(y)-\chi_{E_M}(y)}{|y-q|^{n+s}}\,dy+
\int_{S^c}\frac{\chi_{E_M^c}(y)-\chi_{E_M}(y)}{|y-q|^{n+s}}\,dy\\
&\ge&-C_2 M^{-\frac{s}{4}}
+\int_{S^c}\frac{\chi_{E_M^c}(y)-\chi_{E_M}(y)}{|y-q|^{n+s}}\,dy\\
&\ge&-C_2 M^{-\frac{s}{4}}
-\int_{S^c\cap\{|y_n|\ge M \}}\frac{dy}{|y-q|^{n+s}}
+\int_{S^c\cap\{|y_n|< M \}}\frac{\chi_{E_M^c}(y)-\chi_{E_M}(y)}{|y-q|^{n+s}}\,dy\\&\ge&
-C_2 M^{-\frac{s}{4}}
-\int_{ \{|y-q|\ge M /2\}}\frac{dy}{|y-q|^{n+s}}
+\int_{S^c\cap\{|y_n|< M \}}\frac{dy}{|y-q|^{n+s}}\\&=&
-C_2 M^{-\frac{s}{4}}
-C_3 M^{-s}
+\int_{S^c\cap\{|y_n|< M \}}\frac{dy}{|y-q|^{n+s}},
\end{eqnarray*}
for some~$C_3>0$ depending only on~$n$ and~$s$.

Thus, since~$S^c\cap\{|y_n|< M \}\supset B_1(q+5e_1)$, letting~$C_4:=C_2+C_3$ we have
\begin{eqnarray*}
&&\int_{\capR^n}\frac{\chi_{E_M^c}(y)-\chi_{E_M}(y)}{|y-q|^{n+s}}\,dy
\ge-C_4 M^{-\frac{s}{4}}
+\int_{B_1(q+5e_1) }\frac{dy}{|y-q|^{n+s}}\\&&\qquad
=-C_4 M^{-\frac{s}{4}}
+\int_{B_1(5e_1) }\frac{dy}{|y|^{n+s}}=-C_4 M^{-\frac{s}{4}}+c,
\end{eqnarray*}
for some~$c>0$ depending only on~$n$ and~$s$.
In particular, if~$M$ is sufficiently large, we deduce that the left hand side of~\eqref{ihakjb-iknnsbab66}
is strictly positive, thus reaching a contradiction with~\eqref{ihakjb-iknnsbab66}.

\section{Further remarks on Theorem~\ref{mainTheorem02}}\label{HSI9urj-90369iyktpogmn}

The goal of this section is to stress that the
result in Theorem~\ref{mainTheorem02} is, on the one hand,
related to the classical case of minimal surfaces, since both the classical
and the nonlocal regimes exhibit disconnected minimizers
for large values of~$M$, but, on the other hand,
presents some significant structural differences with respect
to the classical scenario.

More precisely, differently from the classical case,
the minimizer constructed in Theorem~\ref{mainTheorem02}
exhibits the features listed below:

\begin{proposition}\label{KSPN-t3igkrjjghNNSND}
Let~$M$ and~$E_M$ be as in Theorem~\ref{mainTheorem02}. Then,
\begin{equation}\label{AHN-0301}
E_M \supsetneqq\{x_n>M\}\cup\{x_n<-M\}.
\end{equation}
Moreover,
\begin{equation}\label{AHN-0302}
E_M \supset B_{cM^{-s}}(0,\dots,0,-M)\cup B_{cM^{-s}}(0,\dots,0,M),
\end{equation}
for some~$c>0$ depending only on~$n$ and~$s$.

In addition, if~$n=2$, given any~$\epsilon_0>0$ there exists~$c_\star>0$, depending only on~$s$ and~$\epsilon_0$, such that
\begin{equation}\label{AHN-0303}
E_M \supset \left((-1,1)\times\left(-\infty,
-M+c_\star\,M^{-\frac{(2+\epsilon_0)s}{1-s}}
\right) \right)\cup
\left((-1,1)\times\left(M-c_\star\,M^{-\frac{(2+\epsilon_0)s}{1-s}},+\infty\right) \right).
\end{equation}
\end{proposition}

\begin{figure}
	\begin{center}
		\includegraphics[keepaspectratio,scale=0.55,angle=0]{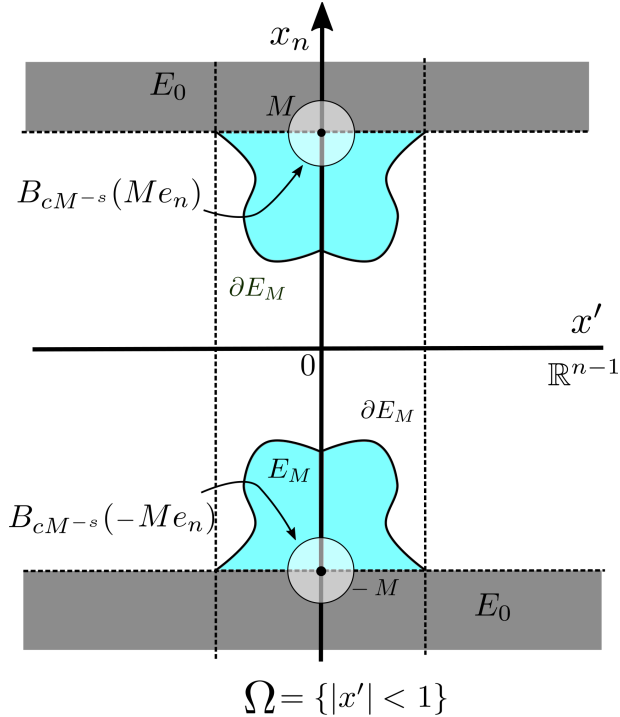}
		\caption{A sketch of an argument in Proposition \ref{KSPN-t3igkrjjghNNSND}.}
		\label{figure04}
	\end{center}
\end{figure}

We remark that~\eqref{AHN-0302}
and~\eqref{AHN-0303} are quantitative versions
of~\eqref{AHN-0301} and a sketch of an argument
used in the proof
of Proposition~\ref{KSPN-t3igkrjjghNNSND} is
depicted in Figure~\ref{figure04}. Though~\eqref{AHN-0302}
and~\eqref{AHN-0303} provide a stronger
result than~\eqref{AHN-0301}, we give an independent proof
of~\eqref{AHN-0301} based on a simple symmetry argument,
while the proofs of~\eqref{AHN-0302} and~\eqref{AHN-0303}
rely on finer quantitative arguments.
We also point out that~\eqref{AHN-0303} provides
an explicit quantitative bound on the stickiness property 
in dimension~$2$.

\begin{proof}[Proof of Proposition~\ref{KSPN-t3igkrjjghNNSND}] To prove~\eqref{AHN-0301},
we need to show that the inclusion in~\eqref{f1} is strict.
For this, we argue by contradiction and suppose that~$E_M =\{x_n>M\}\cup\{x_n<-M\}$.
Then we can use the Euler-Lagrange equation in the viscosity sense shown in \cite[Theorem 5.1]{CRS} 
at the point~$q:=(0,\dots,0,-M)\in\partial E_M$, thus finding that
	\begin{equation}\label{UJONS9uhg}\begin{split}&
	0=	\int_{\capR^n}\frac{\chi_{E_M^c}(y)-\chi_{E_M}(y)}{|y-q|^{n+s}}\,dy=
	\int_{\{|y_n|<M\}}\frac{dy}{|y-q|^{n+s}}
	-
	\int_{\{|y_n|\ge M\}}\frac{dy}{|y-q|^{n+s}}\\&\qquad\qquad=
	\int_{\{z_n\in(0,2M)\}}\frac{dz}{|z|^{n+s}}
	-
	\int_{\{z_n\in(-\infty,0]\cup[2M,+\infty)\}}\frac{dz}{|z|^{n+s}}.
\end{split}
	\end{equation}
Also, by the transformation~$(z',z_n)\mapsto(z',-z_n)$, we see that
$$ \int_{\{z_n\in(0,2M)\}}\frac{dz}{|z|^{n+s}}
=\int_{\{z_n\in(-2M,0)\}}\frac{dz}{|z|^{n+s}},$$
and therefore~\eqref{UJONS9uhg} gives that
$$ 0=-\int_{\{z_n\in(-\infty,-2M]\cup[2M,+\infty)\}}\frac{dz}{|z|^{n+s}}
<0.$$
This contradiction proves~\eqref{AHN-0301},
and we now deal with the proof of~\eqref{AHN-0302}.
To this end, we let~$\phi\in C^\infty_0( \capR^{n-1},\,[0,1])$
with~$\phi(x')=1$ if~$|x'|\le1/2$ and~$\phi(x')=0$ if~$|x'|\ge3/4$. Given~$\eta>0$, we define
$$ F:=\{ x_n <\eta\phi(x')\}$$
and we claim that, for every~$p\in\partial F$,
\begin{equation}\label{KSM-OLS-pel}
\int_{\capR^n}\frac{\chi_{F^c}(y)-\chi_{F}(y)}{|y-p|^{n+s}}\,dy\le C_0\eta,
\end{equation}
for some~$C_0>0$ depending only on~$n$, $s$ and~$\phi$.
To prove this, 
we let
$$ \Psi(x',x_n):=(x',x_n+ \eta\phi(x')) \quad
{\mbox{ and }}\quad\Phi(x):=\Psi(x)-x=(0,\dots,0,\eta\phi(x'))$$
Notice that~$F=\Psi(\{x_n<0\})$ and the Jacobian of~$\Phi$
is bounded by~$C\eta$, together with its derivatives, for some~$C>0$ depending only on~$n$ and~$\eta$.
Furthermore, the inverse of~$\Psi$ is given by
$$ \Psi^{-1}(x)=(x',x_n- \eta\phi(x'))$$
and, setting~$\Xi(x):=\Psi^{-1}(x)-x=-(0,\dots,0,\eta\phi(x'))$,
we find that also the Jacobian of~$\Xi$
is bounded by~$C\eta$. Consequently, we are in the position
of exploiting~\cite[Theorem~1.1]{COZ}
and deduce that
$$
\int_{\capR^n}\frac{\chi_{F^c}(y)-\chi_{F}(y)}{|y-p|^{n+s}}\,dy\le 
\int_{\capR^n}\frac{\chi_{\{y_n>0\}}(y)-\chi_{\{y_n<0\}}(y)}{|y-\Psi^{-1}(p)|^{n+s}}\,dy+C_0\eta
=C_0\eta,$$
for some~$C_0>0$ depending only on~$n$, $s$ and~$\phi$,
thus completing the proof of~\eqref{KSM-OLS-pel}.

Now we define
$$ G:=F\cup \{x_n>4M\},$$
we point out that this union is disjoint
for large~$M$ and small~$\eta$,
and we claim that there exists~$c>0$,
depending only on~$n$, $s$ and~$\phi$,
such that if~$\eta\in(0,cM^{-s}]$ then,
for every~$p\in\partial F$,
\begin{equation}\label{KSM-OLS-pel-2}
\int_{\capR^n}\frac{\chi_{G^c}(y)-\chi_{G}(y)}{|y-p|^{n+s}}\,dy<0.
\end{equation}
Indeed, we have that~$\chi_{G}=\chi_{F}+\chi_{\{x_n>4M\}}$,
whence~$\chi_{G^c}=1-\chi_{G}=1-\chi_{F}-\chi_{\{x_n>4M\}}=
\chi_{F^c}-\chi_{\{x_n>4M\}}$.
Accordingly, we have that~$\chi_{G^c}-\chi_{G}
=\chi_{F^c}-\chi_{F}-2\chi_{\{x_n>4M\}}$
and therefore,
using~\eqref{KSM-OLS-pel},
\begin{eqnarray*}
&& \int_{\capR^n}\frac{\chi_{G^c}(y)-\chi_{G}(y)}{|y-p|^{n+s}}\,dy=
\int_{\capR^n}\frac{\chi_{F^c}(y)-\chi_{F}(y)}{|y-p|^{n+s}}\,dy-
2\int_{\{y_n>4M\}}\frac{dy}{|y-p|^{n+s}}\\&&\qquad\le C_0\eta-
2\int_{(-M,M)^{n-1}\times(4M,5M)}\frac{dy}{|y-p|^{n+s}}\le
C_0\eta-c_0 M^{-s},
\end{eqnarray*}
for some~$c_0>0$ depending only on~$n$ and~$s$,
which plainly leads to~\eqref{KSM-OLS-pel-2}.

By means of~\eqref{KSM-OLS-pel-2}, we can thus use the set~$G$
as a sliding barrier from below with~$\eta:=cM^{-s}$
(starting the sliding from a vertical translation of the set~$G$
equal to~$-2M$)
and find that~$E_M\supset\{
x_n<-M+cM^{-s}\phi(x')\}$. In particular,
we see that~$E_M\supset \left[-\frac12,\frac12\right]^{n-1}\times(-\infty,-M+cM^{-s}]\supset B_{cM^{-s}}(0,\dots,0,-M)$.

Similarly, one proves that~$E_M\supset B_{cM^{-s}}(0,\dots,0,M)$,
thus completing the proof of~\eqref{AHN-0302}.

Now we suppose that~$n=2$ and we establish~\eqref{AHN-0303}.
For this, we fix~$\epsilon_0>0$, we consider a suitably small~$\delta>0$ and
we
exploit~\cite[Corollary 7.2]{DSVboundary} to construct a set~$H\subset\capR^2$ such that
\begin{eqnarray*}
&& H\subset\{x_2<\delta\},\\
&& H\cap\{x_1<-1\}=(-\infty,-1)\times(-\infty,0),\\
&& H\cap\{x_1>1\}=(1,+\infty)\times(-\infty,0),\\
&& H\supset (-1,1)\times\left(
-\infty,\delta^{\frac{2+\epsilon_0}{1-s}}\right)\\
{\mbox{and }}&&
\int_{\capR^2}\frac{\chi_{H^c}(y)-\chi_{H}(y)}{|y-p|^{2+s}}\,dy\le\bar{C}\delta
\end{eqnarray*}
for every~$p=(p_1,p_2)\in \partial H$
with~$|p_1|<1$, where~$\bar{C}>0$ depends only on~$s$ and~$\epsilon_0$.

We define
$$ L:=H\cup \{x_2>4M\},$$
and we see that~$\chi_{L^c}-\chi_{L}
=\chi_{H^c}-\chi_{H}-2\chi_{\{x_2>4M\}}$ and thus
\begin{eqnarray*}
&&\int_{\capR^2}\frac{\chi_{L^c}(y)-\chi_{L}(y)}{|y-p|^{2+s}}\,dy\le
\int_{\capR^2}\frac{\chi_{H^c}(y)-\chi_{H}(y)}{|y-p|^{2+s}}\,dy
-2\int_{\{y_2>4M\}}\frac{dy}{|y-p|^{2+s}}\\&&\qquad\le\bar{C}\delta
-2\int_{(-M,M)\times(4M,5M)}\frac{dy}{|y-p|^{2+s}}\le\bar{C}
\delta-\bar{c}M^{-s}<0
\end{eqnarray*}
for every~$p=(p_1,p_2)\in \partial H$
with~$|p_1|<1$, where~$\bar{c}>0$ depends only on~$s$,
and~$\delta:=\frac{\bar{c}M^{-s}}{2\bar{C}}$. 

In this way, we can use~$L$ as sliding barrier from below
(starting the sliding from a vertical translation of the set~$L$
equal to~$-2M$)
and deduce that
$$E_M\cap\{|x_1|<1\}\supset
\left(
-\infty,-M+\delta^{\frac{2+\epsilon_0}{1-s}}\right)=
\left(
-\infty,-M+c_\star\,M^{-\frac{(2+\epsilon_0)s}{1-s}}\right)$$
for some~$c_\star>0$.
Similarly, one finds that
$$E_M\cap\{|x_1|<1\}\supset
\left(
M-c_\star\,M^{-\frac{(2+\epsilon_0)s}{1-s}},+\infty\right).$$
The proof of~\eqref{AHN-0303}
is thereby complete.\end{proof}

\section*{Acknowledgments}

Supported
by the Australian Research Council DECRA DE180100957
``PDEs, free boundaries and applications''
and the Australian Laureate Fellowship
FL190100081
``Minimal surfaces, free boundaries and partial differential equations''.

The first and third authors are members of INdAM and AustMS.

This work has been carried out during a very fruitful visit
of the second author to the University of Western Australia,
whose warm hospitality is a pleasure to acknowledge.

\end{document}